\documentclass{amsart}
\usepackage[english]{babel}

\usepackage{amsfonts}
\usepackage{amssymb}
\usepackage{color}

\newtheorem{thm}{Theorem}[section]
\newtheorem{cor}[thm]{Corollary}
\newtheorem{lem}[thm]{Lemma}

\theoremstyle{definition}

\theoremstyle{remark}

\numberwithin{equation}{section}

\begin{document}

\title[On Runge type theorems]{On Runge type theorems for solutions to 
strongly uniformly parabolic operators }

\author{A.A. Shlapunov}
\address[Alexander Shlapunov]
{Siberian Federal University
                                                 \\
         pr. Svobodnyi 79
                                                 \\
         660041 Krasnoyarsk
                                                 \\
         Russia}
\email{ashlapunov@sfu-kras.ru} 
\address{Sirius Mathematics Center
Sirius University of Science and Technology, Olimpiyskiy ave. b.1, 354349 Sochi,
         Russia}
\email{shlapunov.aa@talantiuspeh.ru}

\author{P.Yu. Vilkov}
\address[Pavel Vilkov]
{Siberian Federal University
                                                 \\
         pr. Svobodnyi 79
                                                 \\
         660041 Krasnoyarsk
                                                 \\
         Russia}
\email{pavel\_vilkov17@mail.ru}


\begin{abstract}
Let $G_1, G_2 $ be  domains with rather regular boundaries in ${\mathbb R}^{n+1}$, 
$n \geq 2$, such that $G_1 \subset G_2$. We investigate the problem of 
approximation of solutions to strongly uniformly $2m$-parabolic  system $\mathcal L$  
in the domain $G_1$  by solutions to the same system in the  domain 
$G_2$. First, we prove that the space  $S _{\mathcal L}(G_2)$ of solutions to the system 
$\mathcal L$ in the domain $G_2$  is dense in the space $S _{\mathcal L}(G_1)$, endowed with  
the standard Fr\'echet topology of uniform convergence on compact subsets in $G_1$, if 
and only if  the sets $G_2 (t) \setminus G_1 (t)$ have no non-empty compact components 
in  $G_2 (t)$ for each $t\in \mathbb R$, where $G_j (t) = \{x \in {\mathbb R}^n: (x,t) \in 
G_j\}$. Next, under additional  assumptions on the regularity of the bounded domains $G_1$ and 
$G_1(t)$, we prove that solutions from the Lebesgue class $L^2(G_1)\cap S _{\mathcal L}(G_1)$ 
can be approximated by solutions from $S _{\mathcal L}(G_2)$ if and only if  the same 
assumption on the sets $G_2 (t) \setminus G_1 (t)$, $t\in \mathbb R$, is fulfilled.
\end{abstract}

\maketitle

\section*{Introduction}
\label{s.Int}

Apparently, the Approximation Theory in Analysis begins with the famous Weierstra\ss{} 
theorem for continuous functions on segments of the real line, see \cite{W1885}, 
where polynomials were used as the approximating set. It appears that the matter 
has become significantly more complicated for complex functions of complex variable if one 
wants to approximate  them by \textit{holomorphic} polynomials, see \cite{R1885}. 
Actually,  C. Runge \cite{R1885} proposed to use the Cauchy kernel for the so-called rational 
approximation of holomorphic functions on open connected sets (plane domains) 
in $\mathbb C$ in the topology related to uniform convergence on compact subsets of the 
domain. The notion of Runge's pair $\omega \subset \Omega$ of plane domains, for which the 
space ${\mathcal O}(\Omega) $ of holomorphic functions  in $\Omega$ is dense in the space 
$ {\mathcal O}(\omega) $ (endowed with the topology discussed above) gives rise to the 
investigation of topological/geometrical conditions on the pair. The analysis of continuous 
functions approximable on  compact sets  led to the notion of the analytic capacity and the 
creation of the related theory, see, for instance, the paper by A.G. Vitushkin \cite{V67}. 
The matters were successfully extended to the theory of harmonic functions, see 
\cite{Merg56}, and even to solutions to general elliptic operators with the 
uniqueness condition in small, see for instance, \cite{Brw62}, \cite{Lax},
\cite{Mal56}, \cite[ch. 4, 5]{Tark37}, where fundamental  
solutions to the related operators were used instead of the Cauchy kernel. 
 Of course, one should mention the theorem for (non-necessarily elliptic) operators with 
constant coefficients related to the uniform approximation by exponential-polynomial 
solutions in convex domains, see   
\cite{Mal63}, \cite[Ch. VI]{Pal}. By the way, F.E. Browder \cite{Brw62} extended Runge 
type results to some non-elliptic, e.g., hyperbolic, operators.

As for the parabolic equations, the problem of the uniform approximation in the context 
of Runge's pairs was successfully solved for the heat equation in the papers \cite{J}, 
\cite{D80}  (see also \cite{GauTa10} for "the rational approximation" by the functions 
constructed with the use of the heat kernel).

On the other hand,  approximation theorems in various function spaces, where behaviour of the 
elements are controlled up to the boundary of the considered sets, appeared to be more 
important for applications, see, for instance, \cite{Brw62} in the context 
of abstract locally convex topological vector spaces,  \cite{H68} for the analytic 
functions or the monograph \cite[Ch. 5--8]{Tark36} for the Sobolev solutions to systems of 
differential equations with surjective/injective symbols. 
But at the end of the XX-th century interest to this topic in the 
framework of general theory of differential  equations faded. In particular, for 
parabolic equations there are two principal reasons for this 
fact. The first one is the above mentioned theorem on the uniform approximation of solutions 
to operators with the constant coefficients in convex domains by the exponential-polynomial 
solutions covering many needs of approximation results in the Theory of Partial Differential 
Equations. The second reason is that typical problems for parabolic operators (such as the
Cauchy problem in ${\mathbb R}^n \times [0,T)$ or the problems in the cylinder domains 
related to the Dirichlet or Neumann data on the lateral side, combined with the Cauchy data 
on the base of the domain) were handled mostly with the Fourier method of 
separation of variables in the Bochner-Sobolev spaces,  where approximation with respect to 
the functions of the space variables is crucial, see, for example, \cite{Lion69}, 
\cite{Tema79}, or by the integral representation method in the H\"older spaces where the 
approximation is not usually used, see, for instance, \cite{frid}, \cite{LadSoUr67}.

However in recent years Runge's type approximation results were 
established for some classes of non-elliptic operators, including the parabolic ones.
For instance, we mention \cite{EG-FP-S19} on approximation in H\"older spaces of solutions 
on closed sets to 
second order scalar parabolic equations with H\"older coefficients,  
\cite{Klm21} on approximation in Fr\'echet spaces of $C^\infty$-smooth 
solutions for constant coefficients partial differential operators with a single 
characteristic direction, complemented with a theorem on  the so-called 
"quantitative approximation" \cite{DKlm}, and \cite{EP-S21} about approximation of 
solutions to Schr\"odinger equation in Bochner spaces. Also, for 
solutions from the Lebesgue spaces 
 to the heat equation or to the parabolic Lam\'e type operator on cylinder 
domains in ${\mathbb R}^{n+1} $, we refer to recent papers 
\cite{ShHeat} and \cite{VKShLame}, respectively. 

Significantly, many important applications of Runge type results 
for solitions to parabolic operators were discovered at the beginning of the XXI-st 
century, especially to Euler and Navier-Stokes equations of Hydrodynamics and to 
some inverse problems of Mathematical Physics, see 
\cite{EP-S15}, \cite{G-FRZ}, \cite{RueSa}.

Our primary interest is the non-standard Cauchy problem for solutions to the 
parabolic equations in cylinder domains with the Cauchy data on a part of the lateral side of 
the cylinder that behaves much more like the ill-posed Cauchy problem for solutions to 
elliptic systems, see, for instance, \cite{ShTaLMS}, \cite[Ch. 10]{Tark36} for the elliptic 
theory and \cite{KuSh}, \cite{MMT17}, \cite{VKShLame}, \cite{PuSh15}, for the parabolic 
theory in the Sobolev type spaces. The approximation theorems for Runge's pairs in the 
Lebesgue type spaces are crucial for this type of problems because they provide both dense 
solvability  and a  possibility to construct the approximate solutions to the problems, see 
\cite{Merg56}, \cite{MaHa74}, \cite[Theorem 7.6]{ShTaLMS}, and \cite[Corollary 3.5]{VKShLame}, 
respectively. Actually, the non-standard ill-posed Cauchy problem for the parabolic operators 
plays essential role in the development of non-invasive methods of Cardiology, see, for 
instance, \cite[\S 4]{KSU_ZAMM}. Please, note that formulations of such problems do not 
usually use initial data at a suitable time $t=t_0$ and hence we do not need to pose them in 
the cylinder domains.  

Thus, in the present paper we consider approximation theorems for solutions 
to the strongly uniformly parabolic matrix differential operator ${\mathcal L} = \partial _t 
- L$ on a strip ${\mathbb R}^{n} \times \mathcal{I} $, where $\mathcal{I} $ 
is an (open) interval on the time axis, $(-L)$ is a 
strongly elliptic operator with bounded regular coefficients on the strip. We additionally 
assume that both ${\mathcal L}$ and its formal adjoint operator  ${\mathcal L}^*$  admit  
bilateral fundamental solutions and possess the Unique Continuation 
Property with respect to the space variables. The normality property for the fundamental solution 
plays an essential role in the considerations. As far as the topic targets the problems 
without initial data, the presented approach is fit for both the parabolic and backwards 
parabolic operators.  

Section \S \ref{s.1} is devoted to the uniform approximation on compact subsets of the domain 
$G_1 \subset {\mathbb R}^{n+1}$ of the elements of the space $S_{\mathcal L} (G_1)$  
consisting of continuous solutions to the equation ${\mathcal L} u = 0$ in $G_1$ by the 
elements of the space $S_{\mathcal L} (G_2)$ where the domain $G_2$ contains $G_1$. 
We prove that, under reasonable assumptions 
on the regularity of the coefficients of the operator $L$, 
 the space $S_{\mathcal L} (G_2)$ is dense in $S_{\mathcal L} 
(G_1)$ if and only if for each $t \in {\mathcal I}$ the complement of the set $G_1 (t) = \{x \in 
{\mathbb R}^n: (x,t) \in G_1\}$  in the set $G_2 (t)$ has no compact (non-empty) components;  
this is quite similar to the case of the heat equation. 
Actually, taking in accounts a small gap in the proof of the approximation theorem for the  
heat equation in \cite{D80} (that was discovered by \cite{GauTa10}) concerned with rather 
general assumptions on the structure of the domains considered as Runge's pairs in this 
particular situation,  we additionally assume some regularity of the domains' boundaries, see 
assumption $\mathrm{(A)}$, in section \S \ref{s.1}.

In the section \S \ref{s.2} we consider a more subtle problem of approximation 
of solutions to parabolic operator ${\mathcal L}$ from the Lebesgue class 
 $L^2(G_1)$  by more regular solutions in a bigger domain $G_2$. 
We present a solution to this problem in the case where $G_1$ is a bounded 
domain with piece-wise smooth boundary with additional geometric restrictions (see $\mathrm{(A1)}$, 
$\mathrm{(A2)}$). Finally, as a by-product, we obtain the theorem on existence of bases with double 
orthogonality property in Sobolev type spaces of solutions to parabolic systems.  

\section{The uniform approximation}
\label{s.1}

Let ${\mathbb R}^n$, $n \geq 1$, be the $n$-dimensional Euclidean space with the coordinates  
$x=(x_1, \dots , x_n)$ and let  $\Omega \subset {\mathbb R}^n$ be a bounded domain 
(open connected set). As usual, denote by  $\overline{\Omega}$ the closure of $\Omega$, and by
 $\partial\Omega$ its boundary.  
 
We consider functions over  ${\mathbb R}^n$ and 
${\mathbb R}^{n+1}$.  As usual, for $s \in {\mathbb Z}_+$ we denote by $C^s(\Omega)$ 
and $C^s(\overline \Omega)$  the spaces of all $s$ times continuously differentiable 
functions on $\Omega$ and $\overline \Omega$, respectively.  Next,  for  $p\in [0,1)$, 
we denote by $C^{s,p}(\overline \Omega)$ the standard H\"older spaces. 
The spaces $C^{s,p}(\overline \Omega)$ are known to be Banach spaces with the standard  
norms and the  $C^{s,p}(\Omega)$ are the Fr\'echet spaces with the standard semi-norms, see, 
for instance, %\cite{Kry96}, 
\cite{LadSoUr67}.

Let  also $L^2 (\Omega)$ be the Lebesgue space over $\Omega$ with the standard inner product 
$(u,v)_{L^2 (\Omega)} $ and let $H^s (\Omega)$, $s\in \mathbb N$, be the Sobolev space  
with the standard inner product $(u,v)_{H^s (\Omega)}$. As usual, we 
consider the Sobolev space $H^{-s} (\Omega)$, $s\in \mathbb N$, as the dual space 
of $H^{s}_0 (\Omega)$ where $H^{s}_0 (\Omega)$ is the closure of the space 
$C^\infty_{\rm comp} (\Omega)$ consisting of smooth 
functions with compact supports in $\Omega$.  

For a natural number $k$, it is convenient to denote by ${\mathbf C} ^{s,p} _k (\Omega)$ 
the space of $k$-vector functions with components of the class $C ^{s,p} (\Omega)$ and, 
similarly, for the space $\mathbf{L}^2_k (\Omega)$, etc. 

Let $\mathcal{I}$ be a finite or infinite (open) interval on the time axis and let  $L$ be a 
$(k\times k)$-matrix differential operator 
with continuous coefficients in the strip ${\mathbb R}^{n} \times \mathcal{I}$ of an even order 
$2m$:
$$
L = \sum_{|\alpha|\leq 2m} L_{\alpha} (x,t) \partial ^\alpha_x 
$$ 
where $L_{\alpha} (x,t)$ are $(k\times k)$-matrices with entries as described above and 
such that $L^*_\alpha (x,t)= L_\alpha (x,t)$ for all multi-indexes $\alpha \in {\mathbb N}^n$ 
with $|\alpha|=2m$, all $t\in \mathcal{I}$   and all $x \in {\mathbb R}^n$. Consider 
the  strongly 
uniformly (Petrovsky) $2m$-parabolic operator  
$$
{\mathcal L}= \partial _t - L,
$$  
see, for instance, \cite{eid}, \cite{sol}. More precisely, this 
additionally means that the operator $(-L)$ is strongly elliptic, i.e. 
there is a positive constant $c_0$ such that 
$$
(-1)^{m+1} w^* \Big(
\sum_{|\alpha| =2m}L_\alpha (x,t)\zeta ^\alpha \Big) w \geq c_0 |w|^2 |\zeta|^{2mk}
$$
for all $t\in \mathcal{I}$, all $x \in {\mathbb R}^n$, all $\zeta \in {\mathbb R}^n \setminus 
\{0\}$ and all $w \in {\mathbb C}^k \setminus \{0\}$; here $w^*$ is the transposed and complex 
adjoint vector for the complex vector $w \in {\mathbb C}^k$. In particular, for each fixed 
$t_0\in \mathcal{I}$ the operator $L$ is (Petrovsky) elliptic with respect to the space 
variables $x$, i.e. 
$$
\det \Big(\sum_{|\alpha| =2m}L_\alpha (x,t_0)\zeta ^\alpha \Big) \ne 0
$$
for all $x \in {\mathbb R}^n$ and all $\zeta \in {\mathbb R}^n \setminus \{0\}$. 

As usual, we denote by ${\mathcal L}^*$ the formal adjoint operator for ${\mathcal L}$:
$$
{\mathcal L}^*= - \partial _t -  \sum_{|\alpha|\leq 2m} (-1)^{|\alpha|} \partial ^\alpha_x ( 
L^*_{\alpha} (x,t) \, \cdot). 
$$

As we always may proceed with a decomplexification, doubling the dimensions of vectors and matrices, 
but preserving the ellipticity properties of the operator $L$, we assume all the vector functions and coefficients 
of differential operators under the consideration are real-valued.

Following \cite[pp. 245-246]{eid}, we assume also that the operator $\mathcal L$ satisfies 
the following assumptions:
\begin{enumerate}
\item[($\alpha_1$)] 
the coefficients $L_\alpha(x,t)$ are uniformly continuous on the strip 
${\mathbb R}^{n} \times \mathcal{I}$ for $|\alpha|=2m$; 
\item[($\alpha_2$)] 
the coefficients $L_\alpha(x,t)$ are uniformly bounded on the strip 
${\mathbb R}^{n} \times \mathcal{I}$; 
\item[($\alpha_3$)] 
the coefficients $L_\alpha(x,t)$ satisfy a H\"older condition  
with respect to the space 
variables $x$ uniformly on the strip ${\mathbb R}^{n} \times \mathcal{I}$; 
\item[($\alpha_4$)]
each coefficient $L_\alpha(x,t)$ has partial derivatives $\partial^\beta _x L_\alpha(x,t)$ for all 
multi-indeces $\beta$ with $\beta\leq \alpha$  on the strip ${\mathbb R}^{n} \times \mathcal{I}$ satisfying 
assumptions $(\alpha_1)$, $(\alpha_2)$, $(\alpha_3)$.
\end{enumerate}
Under these assumptions on  the coefficients of the operator it admits a unique fundamental 
solution  $\Phi (x,y,t,\tau)$ possessing standard estimates \cite[formulas (2.16), 
(2.17)]{eid}) and the normality property (\cite[Property 2.2]{eid}), i.e. 
\begin{equation} \label{eq.right}
{\mathcal L}_{x,t} \Phi(x,y,t,\tau) = I_k \, \delta (x-y, t-\tau),   
\end{equation}
i.e. the right hand side equals to the unit matrix $I_k$ multiplied by the Dirac 
distribution at the point $(x,t)$ which is commonly written as $\delta (x-y, t-\tau)$, 
where $\delta $ denotes the Dirac distribution at the origin,  and   
\begin{equation} \label{eq.left}
{\mathcal L}^*_{y,\tau} \Phi^*(x,y,t,\tau)  =I_k \delta (x-y, t-\tau), 
\end{equation}
where $\Phi^* = (\Phi_{ji})$ is the adjoint matrix for $\Phi =(\Phi_{ij})$. Note that 
the normality property still holds for complex-valued operators with coefficients satisfying 
assumptions ($\alpha_1$)--($\alpha_4$), but in this case 
$\Phi^* = (\overline \Phi_{ji})$ is the Hermitian adjoint. 

We are going to investigate solutions to $2m$-parabolic equations in  non-cylinder domains 
of special type, see for instance \cite[\S 22]{sol}. Namely let  $G$ be a  
domain in the strip ${\mathbb R}^{n} \times  \mathcal{I}$ and let 
$$
T_1(G)=T_1=\inf_{(x,t)\in G} t, \, T_2(G) = T_2=\sup_{(x,t)\in G} t. 
$$

Consider sets $G (t) = \{x \in {\mathbb R}^n: (x,t) \in G\}$, $t\in 
\mathcal{I}$, playing an essential role in the sequel. We assume that the boundary 
of $G$ satisfies the following property.

\begin{itemize}
\item[(A)] 
The set $G(t)$ is a Lipshitz domain in ${\mathbb R}^n$ for each $t \in (T_1,T_2)$ and
for any numbers $t_3, t_4$ such that $T_1<t_3<t_4<T_2$ 
the set $\Gamma _{t_3,t_4}=\cup_{t\in [t_3,t_4]} \partial G (t)$ is a Lipschitz surface 
in ${\mathbb R}^{n+1}$. 
\end{itemize}

For us, the primary interest for studying parabolic equations in such domains 
is a possibility of applications in the Cardiology, where $G(t) \subset 
{\mathbb R}^3 \times (T_1,T_2)$ is the shape of the human 
myocardium at the time $t$ (see, for instance, \cite{Beheshti2016} for the related 
mathematical models or \cite{KSU_ZAMM} for the particular bi-domain model). 

With this purpose  we need the standard Banach anisotropic spaces $C^{2sm,s} (\overline G)$, see for instance, \cite[\S 22]{sol}, with the norm: 
$$
\|u\|_{C^{2ms,s} (\overline G)} = \sum_{|\alpha|+2mj\leq 2ms} 
\|\partial^\alpha_x \partial^j_t  u\|_{C (\overline G) }.
$$ 
The corresponding anisotropic spaces $C^{2sm,s} (G)$ are Fr\'echet spaces with the topology 
of uniform convergence on compact subsets of $G$ with all the derivatives $\partial_x
^\alpha \partial _t ^j$, $|\alpha|+ 2mj\leq 2ms$, see, for instance, \cite{Sch71}.

Now, $S _{\mathcal L}(G)$ be the set of all the generalized 
$k$-vector functions on $G$, satisfying the (homogeneous)  equation
\begin{equation} \label{eq.heat}
{\mathcal L} u = 0 \mbox{ in } G 
\end{equation}
in the sense of distributions. We endow this space  with the standard  topology of the 
uniform convergence on compact subsets of $G$. Next, we note that estimates  
\cite[formulas (2.16), (2.17)]{eid}) for the fundamental solution imply the standard interior 
a priori estimates for solutions to \eqref{eq.heat}, see, for instance, 
\cite[\S 19]{sol}, or \cite[Ch. 4, \S 2]{frid} for the second order operators. This means 
that all the distributional solutions to equation \eqref{eq.heat} are 
$(2m,1)$-differentiable on 
their domain, i.e. the following continuous embedding holds true:
$$
S _{\mathcal L}(G) \subset \mathbf{C}^{2m,1} _k (G).
$$ 
In particular, this means that $S _{\mathcal L}(G)$ is a closed subspace 
in $\mathbf{C} _k (G)$ and it is a Fr\'echet space itself (with the 
standard Fr\'echet topology inducing the standard uniform convergence together with all the 
derivatives on compact subsets of $G$). 

We also need more assumptions on the operator $\mathcal L$: the Unique Continuation 
Property with respect to the space variables  for  ${\mathcal L}$ and 
${\mathcal L}^*$. Namely, we recall that the Unique Continuation 
Property with respect to the space variables for a differential operator $\tilde L$ on 
a domain $G \subset {\mathbb R}^{n+1}$ consists in the following: 
\begin{enumerate}
\item[(UCP)] for any solution $u \in S _{\tilde {\mathcal L}}(G)$ and any 
$t_0 \in (T_1,T_2)$, if $u (x,t_0)= 0$  for all $x$ from an open 
 subset $\omega \subset G (t_0)$ then $u \equiv 0$ in the open connected component 
of $G (t_0)$, containing $\omega$.
\end{enumerate}

Of course, if ${\mathcal L}$ is the operator with constant coefficients then all the 
assumptions $(\alpha_1)$,   $(\alpha_2)$,  $(\alpha_3)$,  $(\alpha_4)$ and $(UCP)$ 
for both ${\mathcal L}$ and ${\mathcal L}^*$ are 
fulfilled.  Besides, they hold true if the coefficients of the operator $\mathcal L$ are smooth, 
bounded and real analytic with respect to the variables $x$ for each $t\in \mathcal{I}$. In 
this case, the elements of $S_{\mathcal L}(G) $ are actually smooth in $G$ and they are real 
analytic with respect to the space variable $x\in G(t)$ for all $t \in (T_1,T_2)$, see, for 
instance \cite{eid}. Moreover, the following continuous embedding 
holds true in this particular situation:
$$
S _{\mathcal L}(G) \subset \mathbf{C}^{\infty} _k (G).
$$
  
Next, following C. Runge \cite{R1885} and \cite{D80}, we call domains 
$G_1 \subset G_2 \subset {\mathbb R}^{n+1}$ ${\mathcal L}$-Runge's pair 
 if $S_{\mathcal L}(G_2)$ is everywhere dense in $S_{\mathcal L}(G_1)$. 
The following  theorems on Runge's type approximations are quite similar to the corresponding statement 
for the heat equation, see \cite{J} for the case $G_2 = {\mathbb R}^{n+1}$ 
or \cite{D80} for the case of general domains. 
The typical assumption 
in this type of theorems is the following:
\begin{enumerate} 
\item[(B)]
for each $t \in {\mathcal I}$ %the complement of 
the set  $G_2(t)\setminus G_1 (t) $ has no compact (non-empty) components in the set $G_2 (t)$. 
\end{enumerate} 

\begin{thm} \label{t.dense.uniform.suff}
Let ${\mathcal L}$ satisfy assumptions $(\alpha_1)$, $(\alpha_2)$, $(\alpha_3)$, 
$(\alpha_4)$ and $G_1 \subset G_2 $   be domains 
in the strip ${\mathbb R}^n \times {\mathcal I}$ such that $G_2 \ne {\mathbb R}^n \times 
{\mathcal I}$. 
%,satisfying assumption $\mathrm{(A)}$.  
If   $\mathrm{(UCP)}$ is fulfilled 
for  ${\mathcal L}^*$,  and 
the pair $G_1$, $G_2$ satisfies assumptions  $\mathrm{(A)}$ and $\mathrm{(B)}$,
then $S_{\mathcal L}(G_2)$ is everywhere dense in $S_{\mathcal L}(G_1)$. 
%Besides, if
%$\mathrm{(UCP)}$ is fulfilled 
%for  ${\mathcal L}$ and for the 
%coefficients $L_\alpha $, $0<|\alpha|\leq 2m$, the partial derivatives $\partial _t L_{\alpha}$ 
%are continuous on ${\mathbb R}^n \times {\mathcal I}$ then assumption $\mathrm{(B)}$ is neccessary for 
%$S_{\mathcal L}(G_2)$ to be everywhere dense in $S_{\mathcal L}(G_1)$.
\end{thm}

\begin{thm} \label{t.dense.uniform.neces}
Let ${\mathcal L}$ satisfy assumptions $(\alpha_1)$, $(\alpha_2)$, $(\alpha_3)$, 
$(\alpha_4)$ and $G_1 \subset G_2 $   be domains 
in the strip ${\mathbb R}^n \times {\mathcal I}$ satisfying assumption $\mathrm{(A)}$.  If 
% $\mathrm{(UCP)}$ is fulfilled 
%for  ${\mathcal L}^*$, $G_2 \ne {\mathbb R}^n \times {\mathcal I}$, and 
%the pair $G_1$, $G_2$ satisfies assumptions  $\mathrm{(A)}$ and $\mathrm{(B)}$,
%then $S_{\mathcal L}(G_2)$ is everywhere dense in $S_{\mathcal L}(G_1)$. Besides, if
$\mathrm{(UCP)}$ is fulfilled for  ${\mathcal L}$ and for the 
coefficients $L_\alpha $, $0<|\alpha|\leq 2m$, the partial derivatives $\partial _t L_{\alpha}$ 
are continuous on ${\mathbb R}^n \times {\mathcal I}$, then assumption $\mathrm{(B)}$ is neccessary for 
$S_{\mathcal L}(G_2)$ to be everywhere dense in $S_{\mathcal L}(G_1)$.
\end{thm}

Let us proceed with the related proofs.

\begin{proof}[Proof of Theorem \ref{t.dense.uniform.suff}] 
The proof follows the scheme of typical Runge's type theorems related to elliptic operators, cf. \cite{Mal56}, \cite{Brw62}, and their generalizations to parabolic 
ones. To be more precise, we slightly modify  the proof from \cite{D80} for the solutions  
to the heat equation, using the duality theorems from modern functional analysis, 
see, for instance, \cite{Edw65} and the parametrix method for differential equations, 
%classical a priori estimates for parabolic equations, 
see, for example, \cite{eid}, \cite{frid}, \cite{Tark35}. 

As we have noted above, the space $S_{\mathcal L}(G_1)$ is a closed subspace of the space 
${\mathbf C}_k (G_1)$, endowed with the standard topology of the uniform convergence 
on compact subsets of  $G_1$. Then the Hahn-Banach Theorem implies that $G_1$, $G_2$ is a 
${\mathcal L}$-Runge's pair if and only if any continuous functional $F$ on ${\mathbf C}_k 
(G_1)$ annihilating the space $S_{\mathcal L}(G_2)$ also annihilates the space 
$S_{\mathcal L}(G_1)$.

On the other hand, according the Riesz Theorem, see, for instance, 
\cite[Theorem 4.10.1]{Edw65}, any element $F$ of the dual space ${\mathbf 
C}_k^* (G_1)$ for ${\mathbf C}_k (G_1)$
can be presented with the use of a ($k$-vector valued) Radon measure $\mu_F$ with 
compact support $K(\mu_F)$ in $G_1$ , i.e.
\begin{equation} \label{eq.Riesz}
F(u) = %\int_{K(\mu)} u(x,t) ,d\mu_F (x,t) 
\langle u ,d\mu_F  
\rangle \mbox{ for all } u \in S_{\mathcal L}(G_1).
\end{equation}
Let $W$ be the vector function, with components obtained by applying the functional $F$ to the 
corresponding columns of the 
matrix $(x,t) \to \Phi (x,t,y,\tau)$.  For obvious reasons we write 
\begin{equation*}
%\label{eq.v1.A}
W(y,\tau)= %\int_{K(\mu)} 
\langle \Phi^* (x,y, t,\tau) , d\mu _F (x,t) \rangle ;
%=  0 \mbox{ for all } (y,\tau) \not \in G_2.
\end{equation*}
of course, it is well-defined outside the support of $d\mu _F $. 
By  \eqref{eq.right}, for any vector $\varphi \in C^\infty_{k,comp} ({\mathbb R}^{n} \times \mathcal{I})$
we have  
$$
{\mathcal L}_{x,t} \, \langle \Phi^* (x,y, t,\tau) , \varphi (y,\tau) \rangle = \varphi (x,t),
$$
and then, the interior a priori estimates for parabolic systems, see \cite[\S 19]{sol}, 
imply that the vector function 
$V_\varphi(x,t)  = \langle \Phi^* (x,y, t,\tau) , \varphi (y,\tau) \rangle $ belongs to 
$\mathbf{C}^{2m,1}_k ({\mathbb R}^{n} \times \mathcal{I})$. In particular, the vector 
function $W$ can be extended as a distribution to ${\mathbb R}^{n} \times \mathcal{I}$ via 
$$
\langle W, \varphi \rangle = \langle V_\varphi, d\mu_F \rangle, 
$$
because $d\mu_F $ is a ($k$-vector valued) Radon measure with compact support. 

Besides, according to \eqref{eq.right}, columns of the matrix $\Phi (x,t,y,\tau)$ belong to 
$S_{\mathcal L} (G_2)$ with respect to variables $(x,t)$ for each fixed $(y,\tau) \not \in G_2$.
Hence, 
if $F\in {\mathbf C}_k^* (G_1)$ annihilates the space $S_{\mathcal L}(G_2)$ then we have  
\begin{equation}
\label{eq.v1.A}
W(y,\tau)=  0 \mbox{ for all } 
(y,\tau) \not \in G_2.
\end{equation}
But \eqref{eq.left} implies that 
\begin{equation} \label{eq.L*W.1}
{\mathcal L}^* W = d\mu_F \mbox{ in } {\mathbb R}^{n} \times \mathcal{I},  
\end{equation}
in the sense of distributions and, in particular, 
\begin{equation} \label{eq.L*W.2}
{\mathcal L}^* W = 0 \mbox{ in } %{\mathbb R}^{n+1}
\Big({\mathbb R}^{n} \times \mathcal{I} \Big) \setminus K(\mu_F).
\end{equation}
Note that the operator ${\mathcal L}^*$ is backwards-parabolic and, for any 
solution $v(y,\tau)$ to the equation ${\mathcal L}^* v =0$, the vector 
$w(y,\tau)=v(y,-\tau)$  is a solution to the strongly parabolic system of equations  
$(\partial_\tau - L^* _y) w =0$. 
Therefore $W(y,\tau) \in \mathbf{C}^{2m,1}_k
 \big(\big({\mathbb R}^{n} \times \mathcal{I} \big) 
\setminus K(\mu_F)\big)
$ and, 
in particular, it is $C^{2m}$-smooth with respect to  $y$ in ${\mathbb R}^n 
\setminus K(\mu_F) (\tau)$ for each $\tau \in \mathcal{I}$ where 
$K(\mu_F) (\tau) = \{x\in {\mathbb R}^n: (x,\tau) \in K(\mu_F)\}$. 

As both $G_1\subset G_2$ satisfy assumption $\mathrm{(A)}$ and $G_2 \ne {\mathbb R}^n \times {\mathcal I}$, the componets of sets 
${\mathbb R}^n\setminus   G_2 (t) \subset {\mathbb R}^n\setminus G_1 (t)$ are either empty sets 
or closures of Lipschitz domains. 
%CURR
Since the  set $G_2 (t) \setminus G_1 (t)$ has no compact components in $G_2 (t)$, we see that 
 each bounded component of  ${\mathbb R}^{n} \setminus \overline {G_1 (t)} $ intersects with
 ${\mathbb R}^{n} \setminus  \overline {G_2 (t)} $ by a non-empty open set for each $t\in (T_1,T_2)$. 
Hence, by $\mathrm{(UCP)}$ for ${\mathcal L}^*$, the vector $W$ vanishes on every bounded component 
of ${\mathbb R}^{n} \setminus \overline {G_1 (t)} $ for each $t\in (T_1,T_2) $. 
Next, let $\hat G _j (t) $ be the union of $G _j (t)$ with all the 
components of the set $G _j (t)$ that are relatively compact in 
${\mathbb R}^n$. By the discussion above, the closure of $\hat G _1 (t)$ lies in the closure 
of $\hat G _2 (t) $. Then, by De Morgan’s Law in the Sets' Theory we have 
\begin{equation} \label{eq.DeM}
\Big( {\mathbb R}^n \setminus \overline{\hat G _1 (t) }\Big)\cap 
\Big( {\mathbb R}^n \setminus \overline{\hat G _2 (t) }\Big) =   
{\mathbb R}^n \setminus \Big( \overline{\hat G _2 (t) } \cup \overline{\hat G _1 (t) }  \Big)= 
{\mathbb R}^n \setminus  \overline{\hat G _2 (t) }.
\end{equation}
In particular, this means that the vector $W$ vanishes on unbounded components of the set 
${\mathbb R}^n \setminus \overline {G _1 (t)} $ for each $t\in (T_1 (G_1), T_2 (G_1))$, too. 
Thus, \eqref{eq.v1.A} and the Unique Continuation Property $\mathrm{(UCP)}$ 
for the operator ${\mathcal L}^*$ imply that 
\begin{equation*} 
W (y,\tau) = 0   \mbox{ in }{\mathbb R}^{n} \setminus \overline{G_1 (\tau)}
\mbox{ for all } \tau\in \mathcal{I}, 
\end{equation*}
i.e. the vector $W$  is supported in $\overline G_1$. 

As Diaz \cite[p. 644]{D80} noted, complications may arise if  $W$  is not compactly supported 
in $G_1$  even in the case where ${\mathcal L}$ is the heat operator. Nevertheless  he 
considered more general type of domains and for this reason the related proof was rather 
complicated. 

Clearly, for $(x,t) \in \partial G_1$ with $t\in T_1 (G_1), T_2 (G_1))$ it holds 
$x\in \partial G_1(t)$. By hypothesis $\mathrm{(A)}$ and $\mathrm{dist} (\partial G_1, K(\mu (F))>0$ 
it follows immediately that $W$ vanishes in an open subset of the connected component 
of $({\mathbb R}^n \times {\mathcal I})\setminus K(\mu (F)$ in ${\mathbb R}^n \times {\mathcal I}$ which 
contains $(x,t)$. Therefore, by $\mathrm{(UCP)}$ for ${\mathcal L}^* $ and \eqref{eq.L*W.2} it follows that 
$W$ vanishes in a neghbourhood of the point $(x,t)$.  Moreover for 
$(x,t) \in \partial G_1$ with $t\in \{ T_1 (G_1), T_2 (G_1)\}$ it is even simpler to conclude that 
$W$ vanishes in a neghbourhood of the point $(x,t)$. Hence the support $\mathrm{supp} \, (W)$ of $W$ is contained 
in $G_1$. Following \cite{J}, we may even say 
that $W$ is supported in 
$$
\hat K (\mu_F) =  \cup_{\tau \in (T_1,T_2) } \hat K_{G_1}(\mu_F) (\tau), 
%\subset \overline G_1, 
$$  
where $\hat K_{G_1}(\mu_F) (\tau)$ is   
the union of $K(\mu_F) (\tau)$ with all the 
components of the set $G_1 (\tau)\setminus K(\mu_F) (\tau)$ that are relatively compact in 
$G_1 (\tau)$. Then, to prove that $\mathrm{supp} \, (W)\textsc{}$ is compact, we may 
almost literally repeat the arguments from the proof of the crucail lemma in \cite[\S 2]{J}. Indeed, 
each section $K_{G_1}(\mu_F) (\tau) = \emptyset$  for all sufficiently large $|\tau|$. Therefore, 
$\hat K (\mu_F)$ is bounded. To prove that $\hat K (\mu_F)$ is closed, suppose that 
$(x,\tau) \not \in \hat K (\mu_F)$. This means $x \not \in \hat K (\mu_F) (\tau)$, and
this implies that there exists a continuous curve $\gamma$ in ${\mathbb R}^n$ which starts at $x$ 
and tends to $\infty$ lying in the open set ${\mathbb R}^n \setminus \hat K (\mu_F) (\tau)$. 
Choose a closed ball $B \subset {\mathbb R}^n \setminus \hat K (\mu_F) (\tau)$ centered at $x$. 
Then there exists $\varepsilon  > 0$ such that for $|\tau - \tau'| < $ 
the set $K_{G_1}(\mu_F) (\tau')$ is disjoint from $B$ and the image of $\gamma$. But then for $x' \in  B$ the
point $x' \not \in  \hat K (\mu_F) (\tau')$. Thus, $B (x) \times (\tau- \varepsilon, 
\tau + \varepsilon)$ is disjoint from $\hat K (\mu_F)$. This proves ${\mathbf R}^{n+1}
\setminus \hat K (\mu_F)$ is open and then $\hat K (\mu_F)$ is closed, i.e. $\hat K (\mu_F)$ is a 
compact.

Thus, as $W$ is compactly supported in $G_1$  
then we may take  a function $\varphi \in C^\infty (G_2) $ compactly supported in 
$G_1$ such that $\varphi \equiv 1 $ on a neighbourhood of ${\rm supp} \, (W)$. 
Then, by the Leibniz rule, for each $u \in S_{\mathcal L}(G_1)$ we have 
$$
{\mathcal L} (\varphi u) = \varphi {\mathcal L}  u   + 
(\partial_t \varphi) u  + \sum_{ |\alpha|\leq 2m} L_\alpha (x) 
\sum_{\beta + \gamma =\alpha , |\gamma|\geq 1} C_{\beta,\gamma}
(\partial ^{\gamma} \varphi) (\partial ^{\beta} u) = %\psi \tilde L u
$$
$$
(\partial_t \varphi) u  + \sum_{ |\alpha|\leq 2m} L_\alpha (x) 
\sum_{\beta + \gamma =\alpha , |\gamma|\geq 1} C_{\beta,\gamma}
(\partial ^{\gamma} \varphi) (\partial ^{\beta} u), 
$$
where $C_{\beta,\gamma}$ are binomial type coefficients; in particular,  
$$
{\rm supp} \, ({\mathcal L} (\varphi u)) 
\cap {\rm supp} \, (W)  =\emptyset.
$$ 
Therefore \eqref{eq.Riesz} and \eqref{eq.L*W.1} yield
\begin{equation} \label{eq.supported}
F (u) = \langle u ,d\mu_F  \rangle = \langle  u , {\mathcal L}^* W \rangle =
\langle  (\varphi u) , {\mathcal L}^* W \rangle = \langle  {\mathcal L} (\varphi u) ,  W 
\rangle = 0,
\end{equation} 
i.e. $F$ annihilates $S_{\mathcal L}(G_1)$, too. 
\end{proof}

Now we may proceed with the proof of the second theorem.

\begin{proof}[Proof of Theorem \ref{t.dense.uniform.neces}] 
Let $K(t) \Subset G_2 (t)$ stand for compact (possibly, empty) component of $G_2(t) \setminus G_1 (t)$ in $G_2(t)$. 
Taking into the account assumption $\mathrm{(A)}$,  we see that $K(t) \cup G_1 (t)$ is an open subset in $G_2(t)$ 
and $K(t) \Subset K(t) \cup G_1 (t)$ for $t \in (T_1 (G_1), T_2 (G_1))$. 
Let there be a number $t_0\in \mathbb R$ such that the set 
$G_2 (t_0) \setminus G_1 (t_0) $  have a compact non-empty component $K (t_0)$ in the set 
$G_2 (t_0)$. Let us prove that there is a vector $u\in S_{\mathcal L} (G_1)$ that can not be 
approximated by elements of $S_{\mathcal L} (G_2)$. 

With this purpose, fix the oprentation of ${\mathbb R}^{n+1}$ by choosing orders of coordinate 
axes $Ot, Ox_1, \dots Ox_n$.  According to \cite[\S 2.4.2]{Tark35} the operator $L$ 
admits a Green bi-differential 
operator ${\mathcal G}_{ L}$ of order $(2m-1)$ with respect to the space variables $x$, acting from 
$ \mathbf{C}^{2m,1}_k (G_2) \times \mathbf{C}^{2m,1}_k (G_2) $
to the space of $(n+1)$-differential forms with  coefficients from $C^1 (G_2)$, i.e. 
\begin{equation}\label{eq.Green.L}
\int_{\partial G_3 }
 {\mathcal G}_{ L} (g,v) = (Lv,g)_{\mathbf{L}^2 _k(G_3)} -
( v,{L} ^* g)_{\mathbf{L}^2 _k(G_3)} \mbox{ for all } g,v\in  \mathbf{C}^{2m,1}_k (\overline G_3) 
\end{equation}
 and any domain $G_3\Subset G_2$ with piecewise smooth boundary. 
The exposition in \cite[\S 2.4.2]{Tark35} was done in the category of operators with $C^\infty$-smooth 
coefficients, but the construction of ${\mathcal G}_{L}$ is transparent 
and \cite[formula (2.4.12)]{Tark35} provides that 
only partial derivatives $\partial _t L_{\alpha}$ or
$\partial _x^\beta L_{\alpha}$  
with $|\beta| \leq |\alpha|$ of the coefficients $L_\alpha$, $|\alpha|>0$, may be used calculating \eqref{eq.Green.L} 
 (recall that the regularity of 
$\partial _x^\beta L_{\alpha}$ is granted by assumption $(\alpha4)$). Then the Green operator 
for ${\mathcal L}$ is given 
as follows: 
$$
 {\mathcal G}_{ \mathcal L} (g,v) = g^* v dx -  {\mathcal G}_{ L} (g,v),
$$
and, similarly to \eqref{eq.Green.L}, we have  
\begin{equation}\label{eq.Green.parab}
\int_{\partial G_3 }
 {\mathcal G}_{ \mathcal L} (g,v) = ({\mathcal L}v,g)_{\mathbf{L}^2 _k(G_3)} - 
( v,{\mathcal L} ^* g)_{\mathbf{L}^2 _k(G_3)} \mbox{ for all } g,v\in  \mathbf{C}^{2m,1}_k (\overline G_3). 
\end{equation}
Now, fix a point $y_0 \in K_0$. 
Then any vector column $U_j (x,t)$, $1\leq j \leq k$, 
of the fundamental matrix $\Phi  (x,y_0, t,t_0)$ 
belongs to the space $S_{\mathcal L} (G_1)$. 
 
First, we assume that 
\begin{equation} \label{eq.Kt}
\cup_{t\in (T_1 (G_1), T_2 (G_1))} K(t) \Subset  \cup_{t\in (T_1 (G_1), T_2 (G_1))} 
\big( G_1(t) \cup K(t)\big).
\end{equation} 
In this case we may choose a %simply connected 
bounded domain  $G_3$ with 
a piecewise smooth boundary $\partial G_3$ such that $(y_0,t_0) \in G_3 \Subset G_2$ and 
$\partial G_3 \Subset G_1$. 
%and no other compact components belongs to $G_3$.  
If the vector function 
$U_j (x,t)$ can be approximated in ${\mathbf C}^{2m,1}_k  (G_1)$ 
by a sequence $\{ u^{(i)}_j \}_{i\in \mathbb N}$ from the 
space $S_{\mathcal L} (G_2)$ 
then  the sequences of the partial derivatives  
$\{ \partial ^\alpha_x \partial_t^j u^{(i)}_j \}$, $|\alpha|+2mj\leq 2m$, 
converge uniformly on $\partial G_3$. On the other hand,  (the first) Green formula
\eqref{eq.Green.parab} and the normality property \eqref{eq.left} of the fundamental solution $\Phi$ 
imply (the second) Green formula: 
\begin{equation}\label{eq.Green.parab.2}
u^{(i)}_j (x,t) = - \int_{\partial G_3 }
 {\mathcal G}_{ \mathcal L} (\Phi (x,t,y,\tau), u^{(i)}_j (y,\tau)) \mbox{ for all } (x,t) \in G_3.
\end{equation}
Note that there is no need to assume that $G_3$ is a cylinder domain because this Green formula 
is a corollary of the \textit{local} reproducing property of the fundamental solution. 
Now, passing to the limit with respect to $i\to +\infty$ in \eqref{eq.Green.parab.2} we obtain 
\begin{equation}\label{eq.Green.parab.3}
U_j (x,t) = - \int_{\partial G_3 }
 {\mathcal G}_{ \mathcal L} (\Phi (x,t,y,\tau), U_j (y,\tau)) \mbox{ for all } (x,t) \in G_3 \cap G_1.
\end{equation}
However, since $\Phi$ is a fundamental solution to ${\mathcal L}$ then the right-hand side of formula 
\eqref{eq.Green.parab.3} belongs to $S_{\mathcal L} (G_3)$. Therefore 
the vector function $U_j $ extends as a solution $V_j$ to equation \eqref{eq.heat} from $G_1\cap G_3 $ 
to $G_3$, i.e.  to a neighbourhood of the point $(y_0,t_0)$. In particular, assumption 
$\mathrm{(UCP)}$ for the operator ${\mathcal L}$ 
with respect to the space variables implies that this extension is unique on $G_3\setminus (y_0,t_0)$.
This means the vector function $V_j \in S_{\mathcal L} (G_3)$
coincide with the $j$-th vector column $U_j$ of the fundamental matrix $\Phi (x,t,y_0,t_0)$ in 
$G_3 \setminus (y_0,t_0))$.
Thus, we obtain a contradiction because for the matrix $V (x,t)$ with columns $V_j$, $1\leq j \leq k$
 we have ${\mathcal L}V  =0$ in $G_3$  
  but ${\mathcal L} \Phi  (x,y_0,t,t_0)$ coincides with the $\delta$-functional 
concentrated at the point $(y_0,t_0)$. 

Finally, if \eqref{eq.Kt} is not fulfilled we consider the domain
$$
\tilde G_1 = \Big(\cup_{t\in (T_1 (G_1), T_2 (G_1))} 
\big( G_1(t) \cup K(t)\big)\Big) \setminus K (t_0) \subset G_2.
$$
containing the domain $G_1$. By the discussion above, the space 
$S_{\mathcal L} (G_2)$ is not everywhere dense in  $S_{\mathcal L} (\tilde G_1)$  and 
 hence there is functional $F_0\in (S_{\mathcal L} (\tilde G_1))^*$ that annihilates 
the space $S_{\mathcal L} (G_2)$ but does not annihilate $S_{\mathcal L} (\tilde G_1)$. 
On the other hand,  $S_{\mathcal L} (\tilde G_1) \subset 
S_{\mathcal L} (G_1) $ and by Hahn-Banach there is an extension $F_1\in (S_{\mathcal L} (G_1))^*$ 
of the functional $F_0$. Then, by the construction, the functional $F_1$ annihilates 
the space $S_{\mathcal L} (G_2)$ but does not annihilate the space 
$ S_{\mathcal L} (\tilde G_1) \subset S_{\mathcal L} (G_1) $, i.e. 
$S_{\mathcal L} (G_2)$ can not be everywhere dense in  $S_{\mathcal L} (G_1)$. 
\end{proof}

\section{The approximation in the mean}
\label{s.2}

In this section we discuss an approximation theorem for solutions to the operator 
$\mathcal L$ belonging to the Lebesgue spaces. Actually, it is quite similar to the 
approximation theorems for elliptic operators mentioned in the introduction and 
the approximation theorem with uniform convergence on compact subsets for parabolic systems 
proved in the previous section. Also, they are known for the heat equation or for the 
parabolic Lam\'e system in cylinder domains with rather regular lateral surfaces, 
see, for instance,  \cite{ShHeat} or \cite{VKShLame}. 

Investigating spaces of solutions to $2m$-parabolic equation, we need the 
anisotro\-pic Sobolev 
spaces $H^{2ms,s} (G)$, $s \in  {\mathbb Z}_+$, in a domain $G\subset {\mathbb R}^{n} 
\times \mathcal{I}$ with the standard inner product,
$$
(u,v)_{H^{2ms,s} (G)} = \sum_{|\alpha|+2mj\leq 2ms} 
(\partial^\alpha_x \partial^j_t  u, \partial^\alpha_x \partial^j_t v)_{L^2 (G) }.
$$ 
Also, for $\gamma \in {\mathbb Z}_+$, we denote by $H^{\gamma,2sm,s} (G)$ 
the set of all functions  $u \in H^{2sm,s} (G)$ such that 
 $\partial ^\beta_x u \in H^{2ms,s} (G)$ for all  $|\beta|\leq \gamma$.  
As before, it is convenient to denote by 
$\mathbf{H}^{2ms,s}_k (G)$ the space of all the $k$-vector functions with 
the components from ${H}^{2ms,s} (G)$, and similarly for the spaces
$\mathbf{H}^{\gamma, 2ms,s}_k (G)$, etc.

We also will use the so-called Bochner spaces 
of functions depending on $(x,t)$ from the strip  
$\mathbb{R}^n \times  [T_1,T_2]$ with finite numbers $T_1<T_2$. 
Namely, for a Banach space $\mathcal B$ (for example, the space of functions 
on a sub-domain of $\mathbb{R}^n$) and    $p \geq 1$, we denote by  
$L^p (I,{\mathcal B})$ the Banach space of all the measurable mappings 
  $u : [T_1,T_2] \to {\mathcal B}$
with the finite norm  
$$
   \| u \|_{L^p ([T_1,T_2],{\mathcal B})}
 := \| \|  u (\cdot,t) \|_{\mathcal B} \|_{L^p ([T_1,T_2])},
$$
see, for instance, \cite[ch. \S 1.2]{Lion69},  \cite[ch.~III, \S~1]{Tema79}. 

The space $C ([T_1,T_2],{\mathcal B})$ is introduced with the use of the 
same scheme; this is the Banach space of all the continuous mappings
$u : [T_1,T_2] \to {\mathcal B}$ with the finite norm 
$$
   \| u \|_{C ([T_1,T_2],{\mathcal B})}
 := \sup_{t \in [T_1,T_2]} \| u (\cdot,t) \|_{\mathcal B}.
$$

Let ${\mathbf H}^{\gamma,2sm,s} _{k,\mathcal L}(G) =
\mathbf{H}^{\gamma,2sm,s}_k (G) \cap S _{\mathcal L}(G) $, 
$s \in  {\mathbb Z}_+$, $\gamma \in {\mathbb Z}_+$. 
By the discussion in Section \S \ref{s.1}, the space
${\mathbf H}^{2m,1} _{k,\mathcal L}(G)$ is a closed subspace of the Sobolev space 
$\mathbf{H}^{2m,1}_k(G)$. 
Similarly, if  coefficients of the operator $\mathcal L$ are smooth, 
bounded and real analytic with respect to the variables $x$ for each $t\in {\mathcal I}$,  
then ${\mathbf H}^{\gamma,2sm,s} _{k,\mathcal L}(G)$, ${\mathbf C}^{2ms,s} _{k,\mathcal L} 
(\overline{G}) = \mathbf{C}_k^{2ms,s} (\overline{G}) \cap S _{\mathcal L}(G) $, 
${\mathbf C}^{\infty} _{k,\mathcal L} (\overline{G}) =
\mathbf{C}_k^{\infty} (\overline{G}) \cap S _{\mathcal L}(G) $ 
are closed subspaces, consisting of solutions to equation \eqref{eq.heat}, in the spaces 
${\mathbf H}^{\gamma,2sm,s} _{k}(G)$, $\mathbf{C}^{2ms,s}_k  (\overline{G})$ and  
$\mathbf{C}^{\infty}_k  (\overline{G})$, respectively. 

Also, we need the space  $S _{\mathcal L}(\overline{G_2})$, 
defined as follows: 
$$
\cup_{G' \supset 
\overline{G} } S _{\mathcal L}(G'),
$$
where the union is with respect to all the domains $G' \subset {\mathbb R}^{n} \times 
(t_1, t_2)$,  containing the closure of the domain  $G$.
It follows from the a priori estimates referred to  in \S \ref{s.1}
that the following (continuous) embeddings 
\begin{equation} \label{eq.emb.H}
S _{\mathcal L}(\overline{G}) \subset 
\mathbf{C}^{2m,1} _{k,\mathcal L} (\overline{G}) \subset
\mathbf{H}^{2m,1} _{k,\mathcal L}(G) 
\end{equation}
are fulfilled. Of course, if we additionally know that the coefficients of the operator 
${\mathcal L}$ are constant or smooth and real analytic with respect to the space variables 
then the operator is hypoelliptic and the following (continuous) embeddings 
\begin{equation} \label{eq.emb.H.infty}
S _{\mathcal L}(\overline{G}) \subset 
\mathbf{C}^{\infty} _{k,\mathcal L} (\overline{G}) \subset
\mathbf{H}^{\gamma,2ms,s} _{k,\mathcal L}(G) 
\end{equation} 
hold true  $\gamma,s\in {\mathbb Z}_+$. 

To prove an approximation theorem for the spaces of the Lebesgue solutions to 
$\mathcal L$,  we need more regularity of $\partial G$:  
\begin{itemize}
\item[(A1)]    
For each $t\in [T_1, T_2]$, the sets  
$G (t) = \{x \in {\mathbb R}^n: (x,t) \in G\}$, 
are  domains in ${\mathbb R}^n$ with $C^{2m}$-boundaries if $n\geq 2$ 
 or the union of a finite numbers of intervals if $n=1$.
\item[(A2)] 
The boundary $\partial G$ of $G$ is the union $G (T_1)\cup  G (T_2) \cup \Gamma $, where 
$$\Gamma =\cup_{t\in (T_1,T_2)} \partial G (t)$$ 
is a $C^{2m,1}$-smooth surface without  
%no 
points where the tangential planes are parallel to the coordinate plane $\{t=0\}$, i.e. 
we have 
$$
\sum_{j=1}^n  (\nu_j (x,t))^2 \geq \varepsilon _0 \mbox{ for all } (x,t) \in \Gamma
$$
with a positive number $\varepsilon _0$.
\end{itemize}

Under these assumptions we easily see that functions from the space $H^{2m,1} (G)$ and  
some of their partial derivatives have reasonable traces on $\partial G$. The definition, 
the uniqueness and an existence theorem for traces of functions from anisotropic spaces can 
be found in \cite[\S 10]{BeINi}.  In our particular situation we may 
specify the traces  by hands (cf. \cite[Ch.3, \S 7]{MikhX} for anisotropic spaces in cylinder 
domains). With this purpose, we denote by $Y_{s-1/2} (\Gamma)$ the completion   
of ${C}^{s} 
(\Gamma)$, $s \in \mathbb N$, with respect to the norm 
$$
\|\cdot\|_{Y_{s-1/2}(\Gamma)} =\left( \int_{T_1}^{T_2} 
\|\cdot\|^2_{H^{s-1/2} (\partial G (t))} dt \right)^{1/2};
$$
by construction, these are Hilbert spaces embedded continuously 
into ${L} ^2 (\Gamma)$. 

Of course, if $G = \Omega \times (T_1,T_2)$ is a cylinder domain in ${\mathbb R}^{n+1}$
with the base $\Omega$ being a domain with the boundary of class $C^{s}$, then 
$\Gamma = \partial \Omega \times (T_1, T_2)$ and $ Y_{s-1/2} (\Gamma) $ coincides with the Bochner space 
$L^2 ([T_1, T_2], H^{s-1/2} (\partial \Omega))$. 

\begin{lem} \label{l.trace.c0} Let $G$ be a relatively compact domain 
in ${\mathbb R}^n \times {\mathcal I}$ satisfying $\mathrm{(A1)}$, $\mathrm{(A2)}$. 
Then any $w \in {H}^{2m,1} (G)$  
has well-defined trace in $\partial G$. 
Besides, its derivatives  $\partial^\alpha v$ have traces 
on $\Gamma$ of the class $Y_{2m-|\alpha|-1/2} (\Gamma)$  
for each $\alpha \in {\mathbb Z}^n_+$ with $|\alpha|\leq 2m-1$. 
\end{lem}

\begin{proof} First, we note that the space ${H}^{2m,1} (G)$  is continuously 
embedded into the isotropic Sobolev space ${H}^{1} (G)$ and hence the trace
$w_{|\partial G}$ is well-defined. Actually, it belongs to ${H}^{1/2} (\partial G)$, 
see, for instance, \cite{BeINi}. 

Next, by the structure of the domain and the Fubini theorem, 
we have 
$$
\sum_{|\alpha| \leq 2m}\|\partial ^\alpha_y w\|^2_{L^2 (G)} = 
\sum_{|\alpha| \leq 2m} 
\int_{T_1}^{T_2} \int_{G (\tau)} |\partial ^\alpha_y v(w,\tau)|^2 dy  d\tau = 
$$
$$ 
\int_{T_1}^{T_2} \| w(\cdot,\tau)\|^2 _{H^{2m} (G(\tau))}  d\tau 
$$
for any $w\in {C}^{2m} (\overline G)$. 

Now the standard Trace  Theorem for the Sobolev spaces applied for the spaces 
$H^{2m} (G(t))$ yields
$$
\|\partial ^\alpha_y h \|_{H^{2m-|\alpha|-1/2} 
(\partial G (t))}\leq C_{G(t),\alpha} \|\partial ^\alpha_y h \|_{H^{2m-|\alpha|} 
(G(t))}
$$
for all $|\alpha|\leq 2m-1$, for all $h \in H^{2m } (G (t))$ and each $t \in (T_1,T_2)$. 
Hence 
\begin{equation} \label{eq.1}
\int_{T_1}^{T_2} \!\!\! \sum_{|\alpha| 
\leq 2m-1} \!\!\! \|\partial ^\alpha_y w (\cdot ,\tau)\|^{2}_{H^{2m-|\alpha|-1/2} 
(\partial G (t))} d\tau \leq C \! \int_{T_1}^{T_2} \! \| w(\cdot,\tau)\|^2 _{H^{2m} (G(\tau))}  
d\tau 
\end{equation}
for any $w\in {C}^{2m} (\overline G)$; 
here the positive constant 
$$
C=C(G)=\sup_{|\alpha|\leq 2m-1}\sup_{t\in [T_1,T_2]} 
C_{G (t),\alpha}
$$ 
is finite because the related constants $C_{G(t),\alpha} $ may be chosen to depend on 
the $(n-1)$-measure of the domains $\partial G (t)$.

According to 
\cite[\S 14]{BeINi}, any %vector 
function $w\in {H}^{2m,1} (G)$ may be approximated 
by functions from ${C}_{ {\rm comp}}^\infty  ({\mathbb R}^{n+1})$ in the topology 
of the space ${H}^{2m,1} (G)$. Pick a sequence $w^{(s)}\subset {C}
_{{\rm comp}}^\infty  ({\mathbb R}^{n+1})$ approximating $w$ in ${H}^{2m,1} (G)$. 
Then \eqref{eq.1} yields that the sequence $\{ \partial ^\alpha _y w^{(s)} \}$ is 
a Cauchy sequence in ${Y}_{2m-|\alpha|-1/2} (\Gamma)$ for each $\alpha$ with 
$|\alpha|\leq 2m-1$. As ${Y}_{2m-|\alpha|-1/2} (\Gamma)$ is complete we conclude 
that for each $\alpha$ with $|\alpha|\leq 2m-1$ it converges to an element $w_\alpha$, i.e. 
there is a well-defined trace $\partial ^
\alpha _y w_{|\Gamma} = w_\alpha \in {Y}_{2m-|\alpha|-1/2} (\Gamma)$ of the function 
$\partial ^\alpha _y w$ on the surface $\Gamma$. 
\end{proof}

Now we formulate the main result of this section.

\begin{thm}
\label{t.dense.base.suff}
Let ${\mathcal L}$ satisfy assumptions $(\alpha_1)$, $(\alpha_2)$, $(\alpha_3)$, 
$(\alpha_4)$ and $G_1 \subset G_2 $   be domains 
in the strip ${\mathbb R}^n \times {\mathcal I}$ such that 
$G_2 \ne {\mathbb R}^n \times {\mathcal I}$. 
If   $\mathrm{(UCP)}$ is fulfilled 
for  ${\mathcal L}^*$, domain $G_2$ 
satisfies assumption $\mathrm{(A)}$, bounded domain $G_1$  satisfies $\mathrm{(A1)}$, 
$\mathrm{(A2)}$,  and 
the pair $G_1$, $G_2$ satisfies assumption $\mathrm{(B)}$,
then 
$S_{\mathcal L}(\overline G_2)$ is everywhere dense in the space 
$\mathbf{L}^{2} _{k,\mathcal L}(G_1)$.
\end{thm}

\begin{proof} The proof is similar to the proof of Theorem \ref{t.dense.uniform.suff}. However 
we need a more regular boundary for the domain $G_1$. 

Clearly, the set 
 $S_{\mathcal L}(\overline G_2)$ is everywhere dense in
 $\mathbf{L}^{2} _{k,\mathcal L}(G_1)$ if and only if the 
following relations 
\begin{equation} \label{eq.ort}
(u,w)_{\mathbf{L}^2_k (G_1)} = 0 \mbox{ for all } w\in 
S_{\mathcal L}(\overline {G_2})
\end{equation}
means precisely for  the  vector $u\in \mathbf{L}^{2} _{k,\mathcal L}(G_1)$ 
that $u\equiv 0$ in $G_1$. 

As in the proof of Theorem \ref{t.dense.uniform.suff}, we use the fact that the operator ${\mathcal L}$ admits the bilateral 
fundamental solution possessing the normality property \eqref{eq.left}. 

Let for the vector $u\in \mathbf{L}^{2} _{k,\mathcal L}(G_1)$ relation 
\eqref{eq.ort} be fulfilled. Consider an auxiliary vector function
\begin{equation}
\label{eq.v}
v (y,\tau) = \int_{G_1} \Phi^*(x,y,t,\tau) u (x,t)   dx \, dt.  
\end{equation}
Again, the function $v$ is well defined for $(y,\tau) \not \in G_1$, but, arguing as in the proof 
of Theorem 
\ref{t.dense.uniform.suff}, we easily see that it extends as a distribution to the strip 
${\mathbb R}^n \times {\mathcal I}$. 
According to \eqref{eq.right} we have 
$\Phi(x,y,t,\tau) \in S_{\mathcal L}(\overline G_2)$ 
with respect to variables $(x,t)$ for each fixed 
pair  $(y,\tau) \not \in \overline G_2$ and, then  
relation  \eqref{eq.ort}  implies  
\begin{equation} \label{eq.v1}
v (y,\tau) = 0 \mbox{ in } \Big({\mathbb R}^{n} \times \mathcal{I}\Big) \setminus \overline G_2. 
\end{equation}
On the other hand, \eqref{eq.left} yields  
\begin{equation} \label{eq.v0}
{\mathcal L}^*v    = \chi_{G_1} u \mbox{ in } {\mathbb R}^{n} \times \mathcal{I},
\end{equation}
where $\chi_{ G_1}$ is the characteristic function 
of the domain  $G_1$. Obviously, 
$$
{\mathcal L}^*_{y,\tau} v (y,\tau)=0 \mbox{ in  } \Big({\mathbb R}^{n} \times \mathcal{I}\Big) 
\setminus \overline G_1,
$$
and then, by the discussed above  properties of the fundamental solution, the vector function 
 $v$ is $C^{2m,1}_{k}$-vector function in 
$\Big({\mathbb R}^{n} \times \mathcal{I}\Big)  \setminus \overline G_1 $. 

Since both $G_1$ and $G_2$ satisfy assumption $\mathrm{(A)}$ 
then, in the same way as in the proof of Theorem \ref{t.dense.uniform.suff}, assumption 
$\mathrm{(B)}$, formula \eqref{eq.DeM}  
and the Unique Continuation Property $\mathrm{(UCP)}$ 
for the operator ${\mathcal L}^*$ imply that 
\begin{equation} \label{eq.v2}
v (y,\tau) = 0   \mbox{ in } \Big({\mathbb R}^{n} \times \mathcal{I}\Big)  \setminus \overline G_1 .
\end{equation}
 Then, in addition,  \eqref{eq.v0}, \eqref{eq.v2} mean that the vector  
 $v$ is a solution to the Cauchy problem 
$$
\left\{
\begin{array}{lll}
{\mathcal L}^*  v = \chi_{G_1} u \mbox{ in } {\mathbb R}^{n}  \times 
(T_1-\delta,T_2+\delta),\\
 v(y,T_2+\delta) = 0 \mbox{ on } {\mathbb R}^{n} \\
\end{array}
\right.
$$
with a sufficiently small $\delta>0$ such that $[T_1-\delta,T_2+\delta] \subset \mathcal{I}$. 
Taking in account the natural relation between parabolic and backward-parabolic operators and 
using arguments from  \cite[ch. 2, \S 5 theorem 3]{Kry08}, we may conclude that  
$v\in \mathbf{H}^{2m,1} _k ({\mathbb R}^{n} \times (T_1-\delta,T_2+\delta))$ and the 
solution is unique in this class. The regularity of this unique solution to the Cauchy 
problem can be expressed in term of the Bochner classes, too. Namely, 
 $v\in C ( [T_1-\delta,T_2+\delta], \mathbf{H}^{m}_k({\mathbb R}^{n}) )
\cap L^2 ([T_1-\delta,T_2+\delta], \mathbf{H}^{2m}_k ({\mathbb R}^{n}) )$, 
see, for instance, \cite{Lion69}, \cite[ch. 3, \S 1]{Tema79}, where similar linear problems 
for parabolic equations were considered. In particular, the vector function 
 $v$ belongs to the space 
\begin{equation} \label{eq.space.1}
 C ( [T_1-\delta,T_2+\delta], \mathbf{H}^{m}_k({\mathbb R}^{n}) )
%\cap L^2 ([T_1-\delta, T_2+\delta], \mathbf{H}^{2m}_k ({\mathbb R}^{n}) )
\cap \textbf{H}^{2m,1}_k ({\mathbb R}^{n}  \times (T_1-\delta,T_2+\delta))  .
\end{equation}

\begin{lem} \label{l.dense.c0} Any vector $v$ of type \eqref{eq.v}, satisfying 
\eqref{eq.v2}, can be approximated by the vectors from 
$\mathbf{C}^\infty_{k,{\rm comp}}  (G_1)$ in the topology of the Hilbert space 
$\mathbf{H}^{2m,1}_k (G_1)$. 
\end{lem}

\begin{proof} The approximability of Sobolev functions with compact support in a 
domain is closely related to the notion of trace.  
%Of course, as  $\textbf{H}^{2m,1}_k ({\mathbb R}^{n}  \times (T_1-\delta,T_2+\delta)) $
%is embedded continuously into  $\textbf{H}^1_k (G_1)$, we see that $v \in \textbf{H}^1_{k,0}
%(G_1)$, (i.e. the trace $v_{|\partial G}) \in \textbf{H}^{1/2}_{k}(\partial G_1)$ in the 
%sense of isotropic Sobolev spaces and it equals to zero on $\partial G$) and then it can be 
%approximated by the vectors from $\mathbf{C}^\infty_{k,{\rm comp}}  (G_1)$ in the topology of 
%the Hilbert space $\textbf{H}^1_k (G_1)$. However, it is not we actually need. 
%On the other hand, as the boundary of the domain $G_1$ is rather regular, it 
%satisfies 
%\cite[Condition A]{Alb} and we may use results by \cite{Alb} to conclude that $v$ may be 
%approximated in $\mathbf{H}^{2m,1}_k (G_1)$ by vectors from the space 
%${\mathbf C}_{k, {\rm comp}}^\infty  (\overline G_1)$. Again, it is not precisely 
%what we want. 
However, \eqref{eq.space.1} implies that  traces $v_{|\overline{G(T_j)}}$ of $v$ on the 
closures of  $G(T_j)$ belong to $\mathbf{H}^{m}_k(G(T_j)) $, $1\leq j \leq 2$,  and they are 
equal to zero because of \eqref{eq.v2}. Moreover, according to Lemma \ref{l.trace.c0}, the 
traces of $\partial ^\alpha v_{|\Gamma} \in {\mathbf Y}_k^{2m-|\alpha|-1/2} (\Gamma)$ equal 
to zero on $\Gamma$ for all $|\alpha|\leq 2m-1$.

Now the statement of the lemma 
easily follows with the use of the standard regularisation, see, for instance,  
\cite[ch. 3, \S 5, \S 7,]{MikhX} for the isotropic Sobolev spaces, cf.  
see \cite{ShHeat} for cylinder domains. 

Indeed, denote by $h^{(n)}_\delta (x)$ the standard compactly supported function with 
the support in the  ball  $B(x,\delta) \subset {\mathbb R}^n$ with the centre  
at the point  $x$ and of the radius   $\delta>0$:
$$
h^{(n)}_\delta (x) = \left\{
\begin{array}{lll}
0, & \mathrm{if} & |x|\geq \delta, \\
c(\delta) \exp{(1/(|x|^2-\delta^2))}, & \mathrm{if} & |x|<\delta, \\
\end{array}
\right.
$$
where $c(\delta) $ is  the constant providing equality 
$$
\int_{{\mathbb R}^n} h_{\delta}^{(n)}(x) \, dx =1.
$$
Then, as it is well known, for any function $w \in L^1 (K)$ on a measurable compact 
$K$ in ${\mathbb R}^n$, the standard regularisation  
$$
(R ^{(n)}_\delta w) (x) = \int_{{\mathbb R}^{n}} h^{(n)}_\delta (x-y,t-\tau) w (y,\tau) \, dy d\tau 
$$
belongs to the space  $C^\infty_0  ({\mathbb R}^{n})$ for any positive number 
$\delta$ and the support of $R ^{(n)}_\delta w $ lies in a $\delta$-neighbourhood of $K$.    

According to assumption %$(A1)$, 
$\mathrm{(A2)}$ there is a real valued function 
$\rho (x,t) $ of the class $C^{2m,1}$ in a neighbourhood $U$ of the surface $\Gamma$ 
and such that 
$$
\partial G_1 (t) = \{ (x,t) \in {\mathbb R}^n \times [T_1, T_2]: 
\, \rho (x,t)=0\}, \,\, \nabla \rho \ne 0 \mbox{ in } U 
\} .
$$
Hence, for all sufficiently small numbers $\varepsilon>0$ the sets 
$$
G_1^{\varepsilon} = 
\{ (x,t) \in {\mathbb R}^n: \, \rho (x,t)<-\varepsilon\}
$$
are domains with boundaries of class $C^{2m,1}$ and 
$$
G_1 ^{\varepsilon}\Subset G_1 ^{\varepsilon'}\Subset 
G_1, 
$$
if $0<\varepsilon' <\varepsilon$, and 
$$
\lim_{\varepsilon \to +0}\mbox{mes} (G_1 \setminus 
\overline{G_1  ^{\varepsilon }}) = 0.
$$
Moreover the sets $G_1 (t)$ are domains with boundaries of class $C^{2m}$ and  for the 
Lebesgue measure of the domain  $G_1 (t) \setminus \overline{G_1 ^{\varepsilon} (t) }$ 
we have  
$$
G_1 ^{\varepsilon} (t) \Subset G_1 ^{\varepsilon'} (t)\Subset 
G_1 (t), \, t \in [T_1,T_2],  
$$
$$
\lim_{\varepsilon \to +0}\mbox{mes} (G_1 (t)\setminus 
\overline{G_1  ^{\varepsilon} (t) }) = 0,
$$
uniformly with respect to $t\in [T_1,T_2]$. According to \cite[ch. 3, \S 5, lemma 1]{MikhX}, 
there is a positive constant  $C _1 (G_1)$, depending on the square of the surfaces 
$G_1 (T_j) $, $1\leq j\leq 2$, only, and such that 
\begin{equation} \label{eq.est.12}
C_1 (G_1)\, \varepsilon ^{-2}
 \Big(
\int_{T_1}^{T_1+\varepsilon}\| \tilde v \|^2_{L^2 (G (t) )} dt +
\int_{T_2-\varepsilon}^{T_2}\| \tilde v \|^2_{L^2 (G (t) )} dt 
\Big)\leq
\end{equation}
\begin{equation*}  
\int_{T_1}^{T_1+\varepsilon}(\| \tilde v \|^2_{L^2 (G (t) )} +
\| \partial_t \tilde v \|^2_{L^2 (G (t) )} )dt +
\int_{T_2-\varepsilon}^{T_2}(\| \tilde v \|^2_{L^2 (G (t) )} +
\| \partial_t \tilde v \|^2_{L^2 (G (t) )}) dt 
\end{equation*}
for any function $\tilde v \in H^{1} (G_1)$ with zero traces $\tilde v_{|G_1 (T_j)}$
on $G_1 (T_j)$, $1\leq j\leq 2$. 

Similarly, under assumptions $\mathrm{(A1)}$, ${(A2)}$, if  $\partial G(t) \in C^{2m}$ then 
there is a constant  $C _0 (\Gamma)$, depending on the square of the surface 
$ \Gamma$, only, and such that  
\begin{equation} \label{eq.est.11}
\Big(\int_{T_1}^{T_2}\|\partial ^\alpha \tilde v \|^2_{L^2 (G (t) \setminus \overline {G^
\varepsilon}(t))} dt \Big)^{1/2}\leq C_0 (\Gamma)\, \varepsilon  ^{2m-|\alpha|} 
\Big( \int_{T_1}^{T_2} \|\tilde v \|^2_{H^{2m}(G (t) \setminus \overline {G^\varepsilon}(t))} 
dt \Big)^{1/2}
\end{equation}
for any function $\tilde v \in H^{2m,1} (G_1)$ with zero traces 
$\partial^{\alpha}\tilde v_{|\Gamma}$ on $\Gamma$ corresponding to $|\alpha|\leq 2m-1$.

Set 
$$
R_\varepsilon (x,t) = \int_{T_1+\varepsilon/5}^{T_2-\varepsilon/5} 
 h^{(1)}_{\varepsilon/7} (t-\tau) \, \int_{G_1^{\varepsilon/5} (\tau)}  h^{(n)}_{\varepsilon/7} (x-y) \, dy \, d\tau.
$$
Using assumptions 
on $G_1$ we may choose some constants $c_{\alpha}$, $c_{i} $ independent on  $x$ and $t$ 
such that 
\begin{equation} \label{eq.Bochner.1}
0\leq R_\varepsilon  (x,t)\leq 1,
\end{equation}
\begin{equation} \label{eq.Bochner.2}
|\partial^\alpha_x R_\varepsilon (x,t)| \leq c_{\alpha} \, \varepsilon ^{-|\alpha|}, 
\,\, 
|\partial^i_t  R_\varepsilon (x,t)| \leq c_{i} \, \varepsilon ^{-i}, 
\end{equation}
for all $x \in {\mathbb R}^n$, $t \in \mathbb R$, 
$\alpha \in {\mathbb Z}^n_+$, $i \in {\mathbb Z}_+$, 
cf., for instance, \cite[ch. 3, \S 5]{MikhX}. 

Fix a sequence $\{ v_j\} \subset \mathbf{C}^\infty _{0,k} 
({\mathbb R}^{n+1})$, converging to 
$v$ in the space ${\mathbf H}_k^{2m,1} ({\mathbb R}^{n+1} )$ (as we mentioned above, 
the existence of such a sequence follows from \cite[\S 14]{BeINi}). 
Then the functional sequence 
$$
 \{  v_{j,\varepsilon} (x,t)=  R_\varepsilon (x,t)v_j (x,t)\}
$$
lies to ${\mathbf C}^\infty _{k,0} (G_1)$. 

By the triangle inequality, 
\begin{equation} \label{eq.vk.1}
\|v - v_{j,\varepsilon} \|_{{\mathbf H}_k^{2m,1} (G_1)} \leq 
\|v - v_{j} \|_{{\mathbf H}_k^{2m,1} (G_1)} + 
\|v_j- v_{j,\varepsilon} \|_{{\mathbf H}_k^{2m,1} (G_1)} .
\end{equation}
As 
\begin{equation} \label{eq.lim.1}
\lim_{j\to + \infty}\|v - v_{j} \|_{{\mathbf H}_k^{2m,1} (G_1)} =0,
\end{equation}
then we need to estimate the second summand in the right hand side of formula 
\eqref{eq.vk.1}, only. However,  
$$
v _j (x,t)- v_{j,\varepsilon} (x,t) = 
(1-R_\varepsilon (x,t)) v_j (x,t)
$$
and, in particular,  
\begin{equation} \label{eq.eps.0}
v _j (x,t)- v_{j,\varepsilon} (x,t) = 0 \mbox{ for all } (x,t) \in 
\cup_{t=T_1+\varepsilon}^{T_2-\varepsilon}G_1 ^{\varepsilon} (t) .
\end{equation}
Hence, 
\begin{equation} \label{eq.eps}
2^{-1}\|v _j - v_{j,\varepsilon}\|^2_{\mathbf{H}^{2m,1}_k (G_1)} \leq  
\|(1-R_\varepsilon)  \partial_t v_j \|^2
_{\mathbf{L}^{2}_k (G_1)} + 
\end{equation}
$$
\Big\|  (\partial_t R_\varepsilon) 
v_j \Big\|^2 _{\mathbf{L}_k^{2} (\cup_{t\in [T_1,T_1+\varepsilon] \cup [T_2-\varepsilon ,T_2]}
G_1 ^{\varepsilon} (t) )} +
$$
$$
\sum_{ |\alpha| \leq 2m}
\| \partial^{\alpha} ((1-R_\varepsilon)   v_j) \|^2
_{\mathbf{L}^{2}_k (\cup_{t\in [T_1,T_2]} G_1 (t)  \setminus G_1^{\varepsilon} (t))} . 
$$
Then \eqref{eq.est.12}, 
\eqref{eq.est.11},  \eqref{eq.Bochner.1}, \eqref{eq.Bochner.2}, \eqref{eq.eps.0} 
and the Fubini theorem imply that  
\begin{equation} \label{eq.eps.1}
\sum_{ |\alpha| \leq 2m}
\| \partial^{\alpha} (1-R_\varepsilon   v_j) \|^2
_{\mathbf{L}^{2}_k (G_1)}  
\leq 
\end{equation}
$$
\sum_{|\alpha| \leq 2m } 
\|(1-R_\varepsilon)(  \partial^{\alpha} v_j \|^2
_{\mathbf{L}^{2}_k (\cup_{t\in [T_1,T_2]} G_1 (t)  \setminus G_1^{\varepsilon} (t))}  +
$$
$$
\sum_{|\beta+\gamma| \leq 2m , \beta \ne 0} 
\|(\partial^{\beta}R_\varepsilon)(  \partial^{\gamma} v_j \|^2
_{\mathbf{L}^{2}_k (\cup_{t\in [T_1,T_2]} G_1 (t)  \setminus G_1^{\varepsilon} (t))}  \leq 
$$
$$
\sum_{|\alpha| \leq 2m } 
\|(1-R_\varepsilon)(  \partial^{\alpha} v_j \|^2
_{\mathbf{L}^{2}_k (\cup_{t\in [T_1,T_2]} G_1 (t)  \setminus G_1^{\varepsilon} (t))}  +
$$
$$
\sum_{|\beta+\gamma| \leq 2m , \beta \ne 0} 
\varepsilon^{2m-|\gamma+\beta|}\|   v_j \|^2
_{\mathbf{H}^{2m,0}_k (\cup_{t\in [T_1,T_2]} G_1 (t)  
\setminus G_1^{\varepsilon} (t))}  \leq 
$$
$$
\| v\|_{\mathbf{H}^{2m,1}_k (\cup_{t\in [T_1,T_2]}  
G_1 (t) \setminus \overline{G_1^{\varepsilon} (t) })}
$$
with constants $C,\tilde C$, independent on $v$ and $\varepsilon$. 

The boundaries of the domains $\cup_{t \in [T_1,T_1+\varepsilon] } G_1 (t)$ and 
$\cup_{t \in [T_2-\varepsilon,T_2] } G_1 (t)$ are not smooth, but combining results 
\cite[ch. 3, \S 5]{MikhX} related to a function $v$, having the trace vanishing on surfaces  
$G_1(T_1)$ and $G_2 (T_2)$, with bounds  \eqref{eq.Bochner.1}, \eqref{eq.Bochner.2}, 
we see that 
\begin{equation} \label{eq.eps.3}
\| (\partial _t R_\varepsilon) v \|^2
_{\mathbf{L}^{2}_k (\cup_{t\in [T_1,T_1+\varepsilon] \cup [T_2-\varepsilon ,T_2]}
G_1 ^{\varepsilon} (t) )}
\leq 
\end{equation}
$$
\varepsilon^{-1} \cdot \varepsilon 
C \, \Big( \| _t v\|^2
_{\mathbf{L}^{2}_k (\cup_{t\in [T_1,T_1+\varepsilon] \cup [T_2-\varepsilon ,T_2]} 
G_1 ^{\varepsilon} (t) )} + 
\| \partial _t v\|^2
_{\mathbf{L}^{2}_k (\cup_{t\in [T_1,T_1+\varepsilon] \cup [T_2-\varepsilon ,T_2]} 
G_1 ^{\varepsilon} (t) )} \Big)\leq 
$$
$$
\| v\|_{\mathbf{H}^{2m,1}_k (\cup_{t\in [T_1,T_1+\varepsilon] \cup [T_2-\varepsilon,T_1]}  
G_1 (t) \setminus \overline{G_1^{\varepsilon} (t) })}
$$
with a constant $C$, independent on $v$ and  $\varepsilon$.

Besides, according to  \eqref{eq.Bochner.1}, \eqref{eq.Bochner.2}, 
\begin{equation} \label{eq.eps.4}
\|(1-R_\varepsilon)\partial^i_t v \|^2
_{\mathbf{L}_k^{2} (G_1)}
\leq C \| v\|_{\mathbf{H}^{2m,1}_k (\cup_{t\in [T_1,T_2]}  
G_1 (t) \setminus \overline{G_1^{\varepsilon} (t) })}
\end{equation}
with a constant $C$, independent on $v$ and  $\varepsilon$.

Using the continuity of the Lebesgue integral with respect 
to the measure of the integration set, we conclude that 
\begin{equation} \label{eq.lim.2}
\lim_{\varepsilon \to + 0}
\| v\|_{\mathbf{H}^{2m,1}_k (\cup_{t\in [T_1,T_1+\varepsilon] \cup [T_2-\varepsilon,T_1]}  
G_1 (t) )} =0,
\end{equation}
\begin{equation} \label{eq.lim.2A}
\lim_{\varepsilon \to + 0}\| v\|_{\mathbf{H}^{2m,1}_k (\cup_{t\in [T_1,T_2]}  
G_1 (t) \setminus \overline{G_1^{\varepsilon} (t) })} =0.
\end{equation}

Finally, combining estimates \eqref{eq.eps}--\eqref{eq.eps.4} and taking in account 
\eqref{eq.lim.2}, \eqref{eq.lim.2A}, 
we conclude that the statement of the lemma holds true. 
\end{proof}

Next, using lemma \ref{l.dense.c0} and fixing a sequence 
$\{ v_k \} \subset \mathbf{C}^\infty _{k, {\rm comp}} 
(G_1)$ converging to the vector 
$v$ in $\mathbf{H}^{2m,1}_k (G_1)$, we see that  
$$
\|u\|^2_{\mathbf{L}^2_k (G_1)} = 
(u , {\mathcal L}^* v )_{\mathbf{L}^2_k (G_1)} = 
\lim_{i \to + \infty }(u, {\mathcal L}^* v_ i )_{\mathbf{L}^2 _k(G_1)}   
= 0,
$$
because ${\mathcal L} u =0 $ in $G_1$ in the sense of distributions. 
Thus,  $u \equiv 0$ in $G_1$, i.e. the sufficiency is proved.
\end{proof}

Predictably,  statement similar to Theorem \ref{t.dense.uniform.neces} also holds true. 

\begin{thm}
\label{t.dense.base.neces}
Let ${\mathcal L}$ satisfy assumptions $(\alpha_1)$, $(\alpha_2)$, $(\alpha_3)$, 
$(\alpha_4)$ and $G_1 \subset G_2 $   be domains 
in the strip ${\mathbb R}^n \times {\mathcal I}$ 
satisfying assumption $\mathrm{(A)}$. If
$\mathrm{(UCP)}$ is fulfilled 
for  ${\mathcal L}$ and for the 
coefficients $L_\alpha $, $0<|\alpha|\leq 2m$, the partial derivatives $\partial _t L_{\alpha}$ 
are continuous on ${\mathbb R}^n \times {\mathcal I}$ then assumption $\mathrm{(B)}$ is neccessary for 
$S_{\mathcal L}(\overline G_2)$ to be everywhere dense in the space $\mathbf{L}^{2} 
_{k,\mathcal L}(G_1)$.
\end{thm}

\begin{proof}[Proof of Threorem \ref{t.dense.base.neces}] 
By the interior a priori estimates for parabolic systems, the space 
${\mathbf L}^2_{k,\mathcal L}  (G_1) $ is imbedded continouosly 
to $S_{\mathcal L}  (G_1) $. Then we    
actually we may use the same arguments as in the proof of  
Theorem \ref{t.dense.uniform.neces}.
\end{proof}

Next, we obtain the following useful statement.

\begin{cor} \label{c.dense.base0} 
Let the coefficients of the operator $\mathcal L$ be smooth, bounded and real analytic with 
respect to the variables $x$ for each $t\in \mathcal{I}$. Let also $s,\gamma\in 
{\mathbb Z}_+$,  $G_1 \subset G_2 $ be domains in ${\mathbb R}^{n}\times \mathcal{I}$ such that 
$G_2\ne {\mathbb R}^n \times {\mathcal I} $, domain $G_2$ satisfy $\mathrm{(A)}$, and  
bounded domain $G_1$ satisfy $\mathrm{(A1)}$, $\mathrm{(A2)}$. If $\mathrm{(B)}$ holds true  
then the 
space ${\mathbf C}^\infty_{k,\mathcal L}(\overline G_2)$ is everywhere dense in  
${\mathbf L}^{2} _{k,\mathcal L}(G_1)$; in particular, ${\mathbf H}^{\gamma,2s,s} 
_{k,\mathcal L}(G_2)$ is everywhere dense in  
${\mathbf L}^{2} _{k,\mathcal L}(G_1)$ if $G_2$ is bounded, too. 
\end{cor}

\begin{proof} Follows immediately from Theorem \ref{t.dense.base.suff}, 
because of embeddings \eqref{eq.emb.H.infty}. %{eq.emb.H}.
\end{proof}

Finally, Theorem \ref{t.dense.base.suff} allows to prove the existence 
of a basis with the double orthogonality property in the spaces of solutions 
to the operator $\mathcal L$ that is very useful to investigate the non-standard 
ill-posed Cauchy problem for elliptic and parabolic equations, see, 
\cite[Ch. 12]{Tark36}, \cite{ShTaLMS}, \cite{VKShLame}.

\begin{cor} \label{c.bdo} 
Let ${\mathcal L}$ satisfy assumptions $(\alpha_1)$, $(\alpha_2)$, $(\alpha_3)$, 
$(\alpha_4)$, and let $\mathrm{(UCP)}$ hold for both ${\mathcal L}$ and ${\mathcal L}^*$. 
Let  also $G_1 \subset G_2 $ be bounded domains in ${\mathbb R}^{n}
\times \mathcal{I}$, such that 
 $T_1 (G_1)= T_1 (G_2)$, $T_2 (G_1) = T_2 (G_2)$, $G_2$ satisfy $\mathrm{(A)}$ and 
 $G_1$ satisfy $\mathrm{(A1)}$, $\mathrm{(A2)}$. If $\mathrm{(B)}$ holds true 
then there is an 
orthonormal basis  $\{ b_\nu\}$ is the space $\mathbf{H}^{2m,1}_{k,\mathcal L} (G_2)$ such 
that its restriction $\{ b_{\nu|G_1}\}$ to $G_1$ is an orthonormal basis in the space 
$\mathbf{L}^{2}_{k,\mathcal L} (G_1)$. 
\end{cor} 

\begin{proof} By the definition, the space $\mathbf{H}^{2m,1} _{k,\mathcal L} (G_2)$  is 
embedded continuously into the space $\mathbf{L}^2 _{k,\mathcal L} (G_1)$.  We denote by $R$ 
the natural embedding operator  
$$
R: \mathbf{H}^{2m,1} _{k,\mathcal L}(G_2)\to 
\mathbf{L}^2 _{k,\mathcal L}(G_1).
$$ 
As the numbers $T_1$ and $T_2$ are the same for the domains $G_1$ and $G_2$, 
the Unique Continuation Property $\mathrm{(UCP)}$ for the operator  $\mathcal L$ 
with respect to the space variables implies that the operator $R$ is injective.
Besides, it follows from Theorem \ref{t.dense.base.suff} that the range of the operator 
$R$ is everywhere dense in the space ${\mathbf L}^2 _{k,\mathcal L}(G_1)$. 

By the definition, the anisotropic Sobolev space   
$\mathbf{H}^{2m,1} _{k,\mathcal L}(G_2)$ is embedded continuously 
to the isotropic Sobolev space  $\mathbf{H}^{1} _{k,\mathcal L}(G_2)$. Besides,  
by Rellich-Kondrashov theorem the embedding  
$\mathbf{H}^{1} _{k}(G_1) \to {\mathbf L }^2_k (G_1)$ is compact. 
Thus, taking in account the continuous embedding 
$\mathbf{H}^{1} _{k}(G_2) \to {\mathbf H }^1_k (G_1)$, we see that 
the natural embedding operator $R$  is compact. 

Finally,   \cite[example 1.9]{ShTaLMS} implies that the complete system of eigen-vectors 
of the  compact self-adjoint operator
 $R^* R: 
\mathbf{H}^{2m,1} _{k,\mathcal L}(G_2) \to \mathbf{H}^{2m,1} 
_{k,\mathcal L}(G_2)$ 
is the basis looked for; here  $R^*$ is the adjoint operator for $R$ 
in the sense of the Hilbert space theory.
\end{proof}

\begin{cor} \label{c.bdo.s} 
Let the coefficients of the operator $\mathcal L$ be smooth, bounded and real analytic with 
respect to the variables $x$ for each $t\in \mathcal{I}$. Let also $s,\gamma\in 
{\mathbb Z}_+$,  $G_1 \subset G_2 $ be bounded domains ${\mathbb R}^{n}
\times \mathcal{I}$, such that  
$T_1 (G_1)= T_1 (G_2)$, $T_2 (G_1) = T_2 (G_2)$, domain $G_2$ satisfy $\mathrm{(A)}$, and 
  domain $G_1$ satisfy $\mathrm{(A1)}$, $\mathrm{(A2)}$.  If $\mathrm{(B)}$ holds true
then there is an 
orthonormal basis $\{ b_\nu\}$ is the space $\mathbf{H}^{\gamma,2ms,s}_{k,\mathcal L} (G_2)$ 
such that its restriction $\{ b_{\nu|G_1}\}$ to $G_1$ is an orthonormal basis in the space  
$\mathbf{L}^{2} _{k,\mathcal L} (G_1)$. 
\end{cor} 

\bigskip

{\sc Acknowledgments.}
The first author was supported 
by the Krasnoyarsk Mathematical Center and financed by the Ministry of Science and Higher 
Education of the Russian Federation (Agreement No.075-02-2023-936). The 
second author was supported by Supported by the Ministry of Science and Higher Education of 
the Russian Federation, (Agreement 075-10-2021-093, Project MTH-RND-2124).

\end{document}